\documentclass[12pt]{amsart}
\usepackage{amssymb}
\usepackage[all,cmtip]{xy}
\usepackage{adjustbox}
\usepackage{relsize}
\usepackage{mathtools}

\usepackage{microtype}

\makeatletter
\@namedef{subjclassname@2010}{\textup{2010} Mathematics Subject Classification}
\makeatother

\DeclareMathOperator*{\bigast}{\raisebox{-0.6ex}{\scalebox{2.5}{$\ast$}}}

\newtheorem{theorem}{Theorem}[section]
\newtheorem*{theorem*}{Theorem}
\newtheorem{lemma}[theorem]{Lemma}
\newtheorem{proposition}[theorem]{Proposition}
\newtheorem{corollary}[theorem]{Corollary}
\theoremstyle{definition}

\newtheorem{example}[theorem]{Example}

\newtheorem{remark}[theorem]{Remark}

\begin{document}

\title{Equivariant maps between representation spheres}
\author{Zbigniew B{\l}aszczyk}
\author{Wac{\l}aw Marzantowicz}
\author{Mahender Singh}

\address{Faculty of Mathematics and Computer Science, Adam Mickiewicz University of Pozna{\'n}, ul. Umultowska 87, 61-614 Pozna{\'n}, Poland.}
\email{blaszczyk@amu.edu.pl}
\email{marzan@amu.edu.pl}

\address{Indian Institute of Science Education and Research (IISER) Mohali, Sector 81, Knowledge City, SAS Nagar (Mohali), Post Office Manauli, Punjab 140306, India.}
\email{mahender@iisermohali.ac.in}


\begin{abstract}
Let $G$ be a compact Lie group. We prove that if $V$ and $W$ are orthogonal $G$-representations such that $V^G=W^G=\{0\}$, then a $G$-equivariant map $S(V) \to S(W)$ exists provided that $\dim V^H \leq \dim W^H$ for any closed subgroup $H\subseteq G$. This result is complemented by a reinterpretation in terms of divisibility of certain Euler classes when $G$ is a torus.
\end{abstract}

\maketitle

\section{Introduction}

A basic problem in the theory of transformation groups is to find necessary and sufficient conditions for the existence of a $G$-equivariant map between two $G$-spaces. Perhaps the most well-known result in the necessary direction is the celebrated Borsuk--Ulam theorem~\cite{Borsuk}, which states that if $V$ and $W$ are two orthogonal fixed-point free $\mathbb{Z}_2$-representations, then the existence of a $\mathbb{Z}_2$-equivariant map $S(V) \to S(W)$ implies that $\dim V \leq \dim W$. This result has numerous and far reaching generalizations, see e.g. \cite{Matousek}, \cite{Zivaljevic} for an overview. One such generalization, particularly interesting from the point of view of this note, is:

\begin{theorem}[{\cite{Marzantowicz2}}]\label{thm:motivation}
Let $V$ and $W$ be orthogonal representations of $G=(\mathbb{S}^1)^k$ or \mbox{$G=(\mathbb{Z}_p)^{\ell}$}, $p$~a~prime, such that $V^G=W^G=\{0\}$. If there exists a $G$-equivariant map \mbox{$S(V) \to S(W)$}, then
\begin{equation}\label{A}
\dim V^H \leq \dim W^H \textit{ for any closed subgroup $H
\subseteq G$.}\tag{$\ast$}
\end{equation}
\end{theorem}

On the other hand, sufficient conditions for the existence of $G$-equivariant maps between representation spheres have not been investigated nearly extensively. This is our starting point: we prove in Corollary~\ref{main-cor} that~\eqref{A} is sufficient for the existence of a $G$-equivariant map \mbox{$S(V) \to S(W)$} for \textit{any} compact Lie group $G$. It is not a new result in the sense that it can be extracted from the existing literature, see \cite[Chapter II]{Dieck}, although it is rather buried in the text. This, coupled with the fact that the second-named author has been inquired about converses to various versions of the Borsuk--Ulam theorem, makes us believe that it is worthwhile to carefully spell the details out.

A corollary to the discussion above is that if $G$ is a torus or a $p$-torus, then \eqref{A} is equivalent to the existence of a $G$-equivariant map $S(V)\to S(W)$. When $G$ is a torus, we reinterpret this result in terms of divisibility of Euler classes of $V^H$ and~$W^H$ in $H^*(BG;\mathbb{Z})$. This angle of research has been pursued previously in various guises, e.g. by Marzantowicz~\cite{Marzantowicz1} (in~the same setting, for $G$ a compact Lie group) and Komiya~\cite{Komiya1},~\cite{Komiya2} (with $K$-theoretic Euler classes, for $G$ an abelian compact Lie group). However, in each case only the necessary criteria were described.


\section{Preliminaries}\label{prelim}

\subsection{Notation}

Let $G$ be a compact Lie group. If $H \subseteq G$ is a closed subgroup, then~$NH$ denotes the normalizer of $H$ in $G$ and $WH=NH/H$ the Weil group of $H$. Given a $G$-space~$X$, write $\mathcal{O}(X)$ for the set of isotropy groups of $X$. If $H \in \mathcal{O}(X)$, then $(H)$ stands for its conjugacy class, referred to as an \textit{orbit type}. There is a natural partial order on the set of orbit types of $X$, namely:
\[ (H) \leq (K) \textnormal{ if and only if $K$ is conjugate to a subgroup of $H$}. \]

Recall that a finite-dimensional \textit{$G$-complex} is a $G$-space $X$ that possesses a filtration
\[ X_{(0)} \subseteq X_{(1)} \subseteq \cdots \subseteq X_{(n)} = X \] 
by $G$-invariant subspaces, with $X_{(k+1)}$ obtained from $X_{(k)}$ by attaching equivariant cells $D^{k+1}\times G/H$ via $G$-equivariant maps $S^k \times G/H \to X_{(k)}$, $0 \leq k \leq n-1$. The space $X_{(k)}$ is called the \textit{$k$-skeleton} of $X$ and the integer $n$ is the \textit{\textnormal{(}cellular\textnormal{)} dimension} of $X$. 

Observe that if $X$ is a $G$-complex and $H\subseteq G$ is a closed subgroup, then $X^H=$ \mbox{$\{x \in X \,|\, hx=x \textnormal{ for any $h\in H$}\}$}, the \textit{$H$-fixed point set} of $X$, is a $WH$-complex, while $X^{(H)} = \bigcup_{K \in (H)} X^K$ is a $G$-subcomplex of $X$. The cellular dimension of~$X^{(H)}$ as a $G$-complex is equal to the cellular dimension of $X^H$ as a~$WH$-complex. We denote this dimension by $d_H(X)$.

\subsection{Euler classes calculus}\label{sect:Euler}

Let $G \hookrightarrow EG \to BG$ be the universal principal $G$-bundle and $V$ an orthogonal $G$-representation. The \textit{Borel space} $EG \times_G V = (EG\times V)/G$, where the orbit space is taken with respect to the diagonal action, is a vector bundle with base space $BG$ and fibre $V$. Provided that this bundle is $R$-orientable for some ring $R$, its Euler class, denoted $e(V)$, is called the \textit{Euler class of $V$} (over $R$). 

Let $G=\mathbb{T}^k = (\mathbb{S}^1)^k$. Recall that any non-trivial irreducible orthogonal representation of~$G$ is given by 
\[ V_{(\alpha_1,\ldots,\alpha_k)} = V_1^{\alpha_1}\otimes \cdots \otimes V_k^{\alpha_k}, \]
where the tensor product is considered over the field of complex numbers, and: 
\begin{itemize}
\item $V_i$ stands for the irreducible complex $G$-representation corresponding to the projection $G \to \mathbb{S}^1$ onto the $i$-th coordinate, $1\leq i \leq k$,
\item $V^j$ denotes the $j$-th tensor power of a representation $V$,
\item $0\leq \alpha_i$ for any $1\leq i \leq k$.
\end{itemize}
In particular, every non-trivial irreducible orthogonal $G$-representation is complex one-dimensional and admits complex structure. Consequently, the latter is also true for any orthogonal $G$-representation $V$ without a trivial direct summand, and it follows that the corresponding vector bundle $EG\times_G V$ is integrally orientable. 

Now recall that 
\[ H^*(BG;\mathbb{Z}) \cong \mathbb{Z}[t_1,\ldots,t_k], \] 
where $t_i=e(V_i)$ for $1\leq i \leq k$. Using the facts that $e(V\oplus W)=e(V)e(W)$ and, for one-dimensional representations, $e(V\otimes W)=e(V) + e(W)$, we see that the Euler class of $V=\bigoplus_{\alpha}r_{\alpha}V_{\alpha}$ is given by 
\[ e(V)= \prod_{\alpha} (\alpha_1 t_1 + \cdots + \alpha_k t_k)^{r_{\alpha}}. \] 
In particular, $e(V)=0$ if and only if $V$ contains a trivial direct summand.


\section{The existence of equivariant maps for compact Lie groups}\label{section-compact-lie}
Throughout this section $G$ is a compact Lie group. We will be interested in the existence of $G$-equivariant maps between representation spheres. The main result of this section is Corollary \ref{main-cor}, and the main ingredient in its proof is the following fact from equivariant obstruction theory.

\begin{theorem}[{\cite[Chapter II, Proposition 3.15]{Dieck}}]\label{obstruction_easy}
Let $n \geq 1$ be an integer. Suppose that $(X,A)$ is a relative
$G$-complex with a free action on $X\setminus A$ and $Y$ is an  $(n-1)$-connected and $n$-simple $G$-space.
\begin{enumerate}
\item[\textnormal{(1)}] Any $G$-equivariant map $A \to Y$ can be extended over the $n$-skeleton of $X$.
\item[\textnormal{(2)}] Let $f_0$, $f_1 \colon A \to Y$ be $G$-equivariant maps and $\tilde{f}_0$, $\tilde{f}_1 \colon X_{(n)} \to Y$ their extensions. If $f_0$ and $f_1$ are $G$-homotopic, then there exists a $G$-homotopy between $\tilde{f}_0|_{X_{(n-1)}}$ and $\tilde{f}_1|_{X_{(n-1)}}$ extending the one between $f_0$ and $f_1$.
\end{enumerate}
\end{theorem}

As a matter of fact, Theorem \ref{main-thm-abstract} below is also formulated in \cite[Chapter II]{Dieck}, but its proof is spread throughout the text. We provide what we believe to be a more accessible treatment for the convenience of the reader.

\begin{theorem}\label{main-thm-abstract}
Let $X$ be a finite $G$-complex and $Y$ a $G$-space such that $Y^{(H)}$
is non-empty for any minimal orbit type $(H)$ of $X$.
\begin{enumerate}
\item[\textnormal{(1)}] If $Y^H$ is $(d_H(X)-1)$-connected and $d_H(X)$-simple for any $H \in \mathcal{O}(X)$, then there exists a $G$-equivariant map $X \to Y$.
\item[\textnormal{(2)}] If $Y^H$ is $d_{H}(X)$-connected and $(d_{H}(X)+1)$-simple for any $H \in \mathcal{O}(X)$, then any two $G$-equivariant maps $X \to Y$ are $G$-homotopic.
\end{enumerate}
\end{theorem}

\begin{proof}
(1) In order to construct a $G$-equivariant map $f \colon X \to Y$, we will proceed inductively with respect to partial order on the set of orbit types of $X$. 

If $H$ is a representative of a minimal orbit type of $X$, then $X^H$ is a free $WH$-complex. Define a $WH$-equivariant map $(X^H)_{(0)} \to Y^H$ by sending the $0$-cells of~$X^H$ to an arbitrary orbit of $Y^H$ and extend it to a map $f^H \colon X^H \to Y^H$ by means of Theorem \ref{obstruction_easy}. Since $X^{(H)}$ has a single orbit type, it is $G$-homeomorphic to $(G/H) \times_{WH} X^H$ by \mbox{\cite[Chapter II, Corollary 5.11]{Bredon}}. We can therefore saturate $f^H$ to obtain a $G$-equivariant map $X^{(H)} \to Y^{(H)}$ via the composition
\[ X^{(H)} \approx (G/H) \times_{WH} X^H \to (G/H) \times_{WH} Y^H \to Y^{(H)}, \]
where the last map is given by $[gH,y] \mapsto gy$ (see \cite[Chapter II, Corollary 5.12]{Bredon}).

It is straightforward to see that any two distinct minimal orbit types $(H_i)$, $(H_j)$ have $X^{(H_i)}\cap X^{(H_j)} = \emptyset$, thus the above procedure yields a $G$-equivariant map 
\[ \bigcup_{(H)} X^{(H)} \to \bigcup_{(H)} Y^{(H)},\] 
where $(H)$ runs over all minimal orbit types of $X$.

Now choose $K \in \mathcal{O}(X)$ and assume inductively that $f$ is defined on a subcomplex
\[ X^{<(K)} = \bigcup_{(H) < (K)} X^{(H)}.\]
By construction, $f$ takes values in $Y^{<(K)}$. In view of \cite[Chapter I, Proposition 7.4]{Dieck}, $G$-extensions of $X^{<(K)} \to Y^{<(K)}$ to $X^{(K)} \to Y^{(K)}$ are in one-to-one correspondence with $WK$-extensions of $X^{<K} \to Y^{<K}$ to $X^K \to Y^K$. However, the $WK$-action on $X^K\setminus X^{<K}$ is free, hence Theorem \ref{obstruction_easy} applied to the relative complex $\big(X^{K}, X^{<K}\big)$ results in a $WK$-equivariant map $X^K \to Y^K$. There are only finitely many orbit types, hence this process stops after a finite number of steps, producing a $G$-equivariant map $X \to Y$.

(2) Let $H \in \mathcal{O}(X)$ be a representative of a minimal orbit type. Since
$Y^H$ is path-connected, any two $NH$-equivariant maps $WH \to Y^H$ are $G$-homotopic. Therefore any two $G$-equivariant maps $(X^H)_{(0)} \to Y^H$ are also $G$-homotopic. It now suffices to successively apply the second part of Theorem \ref{obstruction_easy} just as above.
\end{proof}

As a consequence, we obtain the following corollary.

\begin{corollary}\label{main-cor}
Let $V$ and $W$ be two orthogonal $G$-representations with $V^G=W^G=\{0\}$. 
\begin{enumerate}
\item[(1)] If $\dim(V^H) \leq \dim(W^H)$ for any $H\in\mathcal{O}\big(S(V)\big)$, then there exists a $G$-equivariant map $S(V) \to S(W)$.
\item[(2)] If, additionally, $G$ is connected and for any $H\in\mathcal{O}\big(S(V)\big)$ we have $\dim WH > 0$, then any two $G$-equivariant maps $S(V) \to S(W)$ are $G$-homotopic.
\end{enumerate}
\end{corollary}

\begin{proof}
Let $H \in \mathcal{O}\big(S(V)\big)$ and note that the cellular dimension of the $G$-complex $S(V)^{(H)}$ is at most $\dim V^H-1$, since this dimension is equal to the dimension of the orbit space $S(V)^{(H)}/G = S(V)^H/\,WH$. On the other hand, the space $S(W)^H$ is non-empty, simple and $(\dim W^H-2)$-connected. Since $\dim V^H\leq \dim W^H$, applying Theorem \ref{main-thm-abstract} concludes the proof.
\end{proof}


\section{Torus equivariant maps between representation spheres}\label{section-torus}

\subsection{}

Let $G=\mathbb{T}^k$. Unless otherwise stated, $V$ and $W$ are assumed to be orthogonal $G$-representations such that $V^G=W^G=\{0\}$. Given a decomposition of $V$ into irreducible components, say $V=\bigoplus_{\alpha \in \mathcal{A}} r_{\alpha}V_{\alpha}$, we introduce the following notation. For any $\alpha \in \mathcal{A}$, 
\begin{itemize} 
\item $K_{\alpha}$ denotes the kernel of $V_{\alpha}$, and 
\item $\mathbb{T}_{\alpha}$ the connected component of identity of $K_{\alpha}$. 
\end{itemize}
Then $K_{\alpha}$ is a $(k-1)$-dimensional subgroup of $G$ and $\mathbb{T}_{\alpha}$ is a $(k-1)$-dimensional torus. Furthermore, let $m_{\alpha}$ be the index of $\mathbb{T}_{\alpha}$ in $K_{\alpha}$. The number $m_{\alpha}$ is in fact the greatest common divisor of the $k$-tuple $\alpha = (\alpha_1,\ldots,\alpha_k)$. In particular, it indicates whether $V_{\alpha}$ is a tensor power of another irreducible $G$-representation $V_{\tilde{\alpha}}$, where $\tilde{\alpha}=(\tilde{\alpha}_1, \ldots, \tilde{\alpha}_k)$ and~$\alpha=m_{\alpha}\tilde{\alpha}$. Let $\tilde{\mathcal{A}}=\{\tilde{\alpha} \,|\, \alpha \in \mathcal{A} \}$ and, for $\tilde{\alpha} \in \tilde{\mathcal{A}}$, define  $\mathcal{H}_{\tilde{\alpha}}^V=\{\alpha \in \mathcal{A} \,|\, m_{\alpha}\tilde{\alpha} = \alpha\}$. Geometrically, $\mathcal{H}_{\tilde{\alpha}}^V$~corresponds to the set of $\alpha \in \mathcal{A}$ such that $\mathbb{T}_{\alpha} =  \mathbb{T}_{\tilde{\alpha}}$. 

\begin{proposition}\label{prop:enlarging_domain}
Let $V=\bigoplus_{\alpha \in \mathcal{A}}r_{\alpha}V_{\alpha}$ and $W=\bigoplus_{\beta \in \mathcal{B}} q_{\beta}V_{\beta}$ be orthogonal $G$-re\-pre\-sen\-ta\-tions such that $\dim V < \dim W$. Any $G$-equivariant map $S(V) \to S(W)$ can be extended to a $G$-equivariant map $S(V') \to S(W)$, where $V'$ is an orthogonal $G$-representation such that $V \subseteq V'$ and $\dim V' = \dim W$. 
\end{proposition}

\begin{proof}
Let $S(V) \to S(W)$ be a $G$-equivariant map. Note that since $V^G=W^G=\{0\}$, we have $V=\bigoplus_{\tilde{\alpha} \in \tilde{\mathcal{A}}} V^{\mathbb{T}_{\tilde{\alpha}}}$ and $W=\bigoplus_{\tilde{\beta} \in \tilde{\mathcal{B}}} W^{\mathbb{T}_{\tilde{\beta}}}$. In view of Theorem \ref{thm:motivation}, $\dim V^{\mathbb{T}_{\tilde{\alpha}}} \leq \dim W^{\mathbb{T}_{\tilde{\alpha}}}$ for any $\tilde{\alpha} \in \tilde{\mathcal{A}}$, which shows that $\tilde{\mathcal{A}} \subseteq \tilde{\mathcal{B}}$. Furthermore,
\[ \sum_{\tilde{\alpha} \in \tilde{\mathcal{A}}} \dim V^{\mathbb{T}_{\tilde{\alpha}}} = \dim V < \dim W = \sum_{\tilde{\beta} \in \tilde{\mathcal{B}}} \dim W^{\mathbb{T}_{\tilde{\beta}}}.\] 
Consequently, $\tilde{\mathcal{A}} \subsetneq \tilde{\mathcal{B}}$ or $\dim W^{\mathbb{T}_{\tilde{\alpha}}} - \dim V^{\mathbb{T}_{\tilde{\alpha}}} = d_{\tilde{\alpha}} > 0$ for some $\tilde{\alpha} \in \tilde{\mathcal{A}}_1 \subseteq \tilde{\mathcal{A}}$. Set 
\[ V'=V \oplus \Big( \!\bigoplus_{\tilde{\alpha} \in \tilde{\mathcal{A}}_1} d_{\tilde{\alpha}}V_{\tilde{\alpha}}\Big) \oplus \Big(\!\!\bigoplus_{\tilde{\beta} \in \tilde{\mathcal{B}}\setminus \tilde{\mathcal{A}}} (\dim W^{\mathbb{T}_{\tilde{\beta}}}) V_{\tilde{\beta}}\Big). \]
Then $\dim V'=\dim W$ and $\dim\,(V')^H \leq \dim W^H$ for any $H \in \mathcal{O}\big(S(V')\big)$. Indeed, if $H$ properly contains a $(k-1)$-dimensional torus, then $\dim\,(V')^H = \dim V^H$. Otherwise, since $\tilde{\mathcal{A}}' = \tilde{\mathcal{B}}$,
\[ (V')^H = \Big(\bigoplus_{\tilde{\beta} \in \tilde{\mathcal{B}}} (V')^{\mathbb{T}_{\tilde{\beta}}}\Big)^H = \bigoplus_{\tilde{\beta} \in \tilde{\mathcal{B}}} \big((V')^{\mathbb{T}_{\tilde{\beta}}}\big)^H = \!\! \bigoplus_{\substack{\tilde{\beta} \,:\, \tilde{\beta} \in \mathcal{\tilde{B}}\\
\phantom{{}=\tilde{\beta} \,:} H\subseteq \mathbb{T}_{\tilde{\beta}}}} \!\! (V')^{\mathbb{T}_{\tilde{\beta}}}. \]
But $\dim\,(V')^{\mathbb{T}_{\tilde{\beta}}} = \dim W^{\mathbb{T}_{\tilde{\beta}}}$ for any $\tilde{\beta} \in \tilde{\mathcal{B}}$ by construction, hence $\dim\,(V')^H = \dim W^H$. The existence of a $G$-equivariant map $S(V') \to S(W)$ now follows from Corollary \ref{main-cor}.
\end{proof}

\begin{lemma}\label{lemma:divisibility}
If there exists a $G$-equivariant map $S(V) \to S(W)$, then $e(V)$ divides $e(W)$ in $H^*(BG;\mathbb{Z})$.
\end{lemma}

\begin{proof}
In view of Theorem \ref{thm:motivation}, $\dim V \leq \dim W$. If $\dim V < \dim W$, use Proposition \ref{prop:enlarging_domain} to obtain a $G$-equivariant map $S(V') \to S(W)$, where $V'$ is an orthogonal $G$-representation such that $V\subseteq V'$ and $\dim V' = \dim W$. In view of \cite[Proposition 1.8]{Marzantowicz1}, $e(V')$ divides $e(W)$. Since $e(V')=e(V)e(V^{\perp})$, where $V^{\perp}$ is the orthogonal complement of $V$ in $V'$, we see that $e(V)$ also divides $e(W)$.
\end{proof}

\begin{theorem}\label{sufficient for torus}
Let $V=\bigoplus_{\alpha \in \mathcal{A}}r_{\alpha}V_{\alpha}$ and $W=\bigoplus_{\beta \in \mathcal{B}} q_{\beta}V_{\beta}$ be orthogonal $G$-representations. The following conditions are equivalent.
\begin{enumerate}
\item There exists a $G$-equivariant map $S(V)\to S(W)$.
\item For any $H \in \mathcal{O}\big(S(V)\big)$, the Euler class of $V^H$ divides the Euler class of $W^H$ in $H^*(BG;\mathbb{Z})$.
\item For any $H \in \mathcal{O}\big(S(V)\big)$, $\dim V^H \leq \dim W^H$.
\item For any $(k-1)$-dimensional isotropy subgroup $H \in \mathcal{O}\big(S(V)\big)$, $\dim V^H \leq \dim W^H$.
\item For any $\tilde{\alpha} \in \tilde{\mathcal{A}}$ and any $m \in \mathbb{N}$,
\[ \sum_{\substack{\alpha\,:\,\alpha \in \mathcal{H}_{\tilde{\alpha}}^V \\
\phantom{\alpha\,:\,}m \,|\,m_{\alpha}}} \!\! r_{\alpha} \;\; \leq \sum_{\substack{\beta\,:\,\beta \in \mathcal{H}_{\tilde{\alpha}}^{W} \\ \phantom{\beta\,:\,}m \,|\, m_{\beta}}} \!\! q_{\beta}.\]
\end{enumerate}
\end{theorem}

\begin{proof}
``$(1) \Rightarrow (2) \Rightarrow (3) \Rightarrow (1)$''. Let $S(V) \to S(W)$ be a $G$-equivariant map. Choose a subgroup $H \in \mathcal{O}\big(S(V)\big)$ and restrict to a $G$-equivariant map \mbox{$S(V^H) \to S(W^H)$}. It follows from Lemma \ref{lemma:divisibility} that $e(V^H)$ divides $e(W^H)$ in $H^*(BG;\mathbb{Z})$. If this happens, then $\deg e(V^H) \leq \deg e(W^H)$, which directly translates into $\dim V^H \leq \dim W^H$. This last condition for any $H \in \mathcal{O}\big(S(V)\big)$ implies the existence of a $G$-map $S(V) \to S(W)$ via Corollary \ref{main-cor}.

``$(4) \Rightarrow (3)$''. Assume without loss of generality that $k\geq 2$. As exhibited in the proof of Proposition \ref{prop:enlarging_domain}, for $H \in \mathcal{O}\big(S(V)\big)$ at most $(k-2)$-dimensional,
\[ V^H=\!\! \bigoplus_{\substack{\tilde{\alpha} \,:\, \tilde{\alpha} \in \mathcal{\tilde{A}}\\
\phantom{{}=\tilde{\alpha} \,:} H\subseteq \mathbb{T}_{\tilde{\alpha}}}} \!\! V^{\mathbb{T}_{\tilde{\alpha}}} \textnormal{ and } W^H = \!\! \bigoplus_{\substack{\tilde{\beta} \,:\, \tilde{\beta} \in \mathcal{\tilde{B}}\\
\phantom{{}=\tilde{\beta} \,:} H\subseteq \mathbb{T}_{\tilde{\beta}}}} \!\! W^{\mathbb{T}_{\tilde{\beta}}}.\]
Since $\dim V^{K_{\alpha}} \leq \dim W^{K_{\alpha}}$ for any $\alpha \in \mathcal{A}$, we infer that $\tilde{\mathcal{A}} \subseteq \tilde{\mathcal{B}}$. Thus in order to wrap this part of the proof up, it suffices to observe that $\dim V^{\mathbb{T}_{\tilde{\alpha}}} \leq \dim W^{\mathbb{T}_{\tilde{\alpha}}}$ for any $\tilde{\alpha} \in \tilde{\mathcal{A}}$. Indeed, $H=\bigcap_{\alpha \in \mathcal{H}_{\tilde{\alpha}}^V} K_{\alpha}$ is a $(k-1)$-dimensional isotropy of $S(V)$ such that $V^{\mathbb{T}_{\tilde{\alpha}}} = V^H$, hence 
\[ \dim V^{\mathbb{T}_{\tilde{\alpha}}} = \dim V^H \leq \dim W^H \leq \dim W^{\mathbb{T}_{\tilde{\alpha}}}. \]

``$(4) \Leftrightarrow (5)$''. Note that if we view $V^{\mathbb{T}_{\tilde{\alpha}}}$ and $W^{\mathbb{T}_{\tilde{\alpha}}}$ as representations of $\mathbb{S}^1 = G/\mathbb{T}_{\tilde{\alpha}}$, then~$(4)$ can be rephrased as $\dim\, (V^{\mathbb{T}_{\tilde{\alpha}}})^{\mathbb{Z}_m} \leq \dim\, (W^{\mathbb{T}_{\tilde{\alpha}}})^{\mathbb{Z}_m}$ for any $\tilde{\alpha} \in \mathcal{A}$ and $m \in \mathbb{N}$. But
\[ (V^{\mathbb{T}_{\tilde{\alpha}}})^{\mathbb{Z}_m} =   \!\! \bigoplus_{\substack{\alpha \,:\,\alpha \in \mathcal{H}_{\tilde{\alpha}}^V \\
\phantom{\alpha\,:\,}m \,|\,m_{\alpha}}} \!\! r_{\alpha}V_{\alpha}, \]
which shows that 
\[ \dim\,(V^{\mathbb{T}_{\tilde{\alpha}}})^{\mathbb{Z}_m} = \!\!  \sum_{\substack{\alpha\,:\,\alpha \in \mathcal{H}_{\tilde{\alpha}}^V \\
\phantom{\alpha\,:\,} m \,|\,m_{\alpha}}} \!\! r_{\alpha}.\] 
An analogous thing happens for $W$, which concludes the proof.
\end{proof}

\begin{remark}
The implication ``$(4) \Rightarrow (3)$'' can be seen in a more geometrical manner. As observed above, (4) amounts precisely to the condition \eqref{A} for $V^{\mathbb{T}_{\tilde{\alpha}}}$ and $W^{\mathbb{T}_{\tilde{\alpha}}}$ viewed as representations of $\mathbb{S}^1=G/\mathbb{T}_{\tilde{\alpha}}$, for any $\tilde{\alpha} \in \tilde{\mathcal{A}}$. Therefore Corollary \ref{main-cor} implies the existence of an $\mathbb{S}^1$-equivariant map $f^{\mathbb{T}_{\tilde{\alpha}}} \colon S(V^{\mathbb{T}_{\tilde{\alpha}}}) \to S(W^{\mathbb{T}_{\tilde{\alpha}}})$, which can be considered as a $G$-equivariant map. Consequently, the join construction 
\[ S(V) = S\Big(\bigoplus_{\tilde{\alpha} \in \tilde{\mathcal{A}}} V^{\mathbb{T}_{\tilde{\alpha}}}\Big) = \bigast_{\tilde{\alpha} \in \tilde{\mathcal{A}}}
S(V^{\mathbb{T}_{\tilde{\alpha}}})\longrightarrow \bigast_{\tilde{\alpha} \in \tilde{\mathcal{A}}} S(W^{\mathbb{T}_{\tilde{\alpha}}}) = S\Big(\bigoplus_{\tilde{\alpha} \in \tilde{\mathcal{A}}} W^{\mathbb{T}_{\tilde{\alpha}}}\Big) \subseteq S(W)\]
yields the desired $G$-equivariant map.

On a related note, the implication ``$(5) \Rightarrow (2)$'' is a purely algebraic fact and can be derived directly, without any geometrical interpretation. We would like to thank A.~Schinzel for suggesting the following argument to us. 

Suppose that $(5)$ is satisfied and fix $\tilde{\alpha} \in \tilde{\mathcal{A}}$. We will show that 
\[ e(V^{\mathbb{T}_{\tilde{\alpha}}}) = \prod_{\alpha \in \mathcal{H}_{\tilde{\alpha}}^V} (\alpha_1t_1 + \cdots + \alpha_k t_k)^{r_{\alpha}} = \Big(\!\!\prod_{\alpha \in \mathcal{H}_{\tilde{\alpha}}^V} m_{\alpha}^{r_{\alpha}}\Big) (\tilde{\alpha}_1t_1 + \cdots + \tilde{\alpha}_k t_k)^{\sum_{\alpha \in \mathcal{H}_{\tilde{\alpha}}^V}r_{\alpha}}\] 
divides 
\[ e(W^{\mathbb{T}_{\tilde{\alpha}}}) = \Big(\!\!\prod_{\beta \in \mathcal{H}_{\tilde{\alpha}}^{W}} m_{\beta}^{q_{\beta}}\Big) (\tilde{\alpha}_1t_1 + \cdots + \tilde{\alpha}_k t_k)^{\sum_{\beta \in \mathcal{H}_{\tilde{\alpha}}^{W}}q_{\beta}}.\]
If $m=1$, then $(5)$ yields $\sum_{\alpha \in \mathcal{H}_{\tilde{\alpha}}^V} r_{\alpha} \leq \sum_{\beta \in \mathcal{H}_{\tilde{\alpha}}^{W}} q_{\beta}$, thus it suffices to prove that $\prod_{\alpha \in \mathcal{H}_{\tilde{\alpha}}^V} m_{\alpha}^{r_{\alpha}}$ divides $\prod_{\beta \in \mathcal{H}_{\tilde{\alpha}}^{W}} m_{\beta}^{q_{\beta}}$.

Let $n$ be the highest power of a prime $p$ appearing in any of $m_{\alpha}$'s. Observe that $p$ appears in $\prod_{\alpha \in \mathcal{H}_{\tilde{\alpha}}^V} m_{\alpha}^{r_{\alpha}}$ with the power
\begin{align*}
M &= \sum_{\substack{p^{\phantom{2}}\,\mid\,m_{\alpha}\\
p^2\,\nmid\,m_{\alpha}}} \!r_{\alpha} \;+\; 
2\!\sum_{\substack{p^2\,\mid\,m_{\alpha}\\
p^3\,\nmid\,m_{\alpha}}} \!r_{\alpha} \;+\; \cdots \;+\; n \!\sum_{\substack{p^n\,\mid\,m_{\alpha}}} \!r_{\alpha} \\
&= \sum_{\substack{p\,\mid\,m_{\alpha}\\
\phantom{p^2\,\mid\,m_{\alpha}}}} \!r_{\alpha} \;+ \sum_{\substack{p^2\,\mid\,m_{\alpha}\\
p^3\,\nmid\,m_{\alpha}}} \!r_{\alpha} \;+\;
2 \! \sum_{\substack{p^3\,\mid\,m_{\alpha}\\
p^4\,\nmid\,m_{\alpha}}} \!r_{\alpha} \;+\; \cdots \;+\; (n-1) \!\sum_{p^n\,\mid\,m_{\alpha}} \!r_{\alpha} \;=\; \cdots \\
&= \sum_{\substack{p\,\mid\,m_{\alpha}\\
\phantom{p^2\,\mid\,m_{\alpha}}}} \!r_{\alpha} \;+ \sum_{p^2\,\mid\,m_{\alpha}} \!r_{\alpha} \;+\; \cdots \;+ \sum_{p^n\,\mid\,m_{\alpha}} \!r_{\alpha},
\end{align*}
where $\alpha$ varies over $\mathcal{H}_{\tilde{\alpha}}^V$. Likewise, if $m$ is the highest power of $p$ appearing in any of $m_{\beta}$'s, then $p$ appears in $\prod_{\beta \in \mathcal{H}_{\tilde{\alpha}}^{W}} m_{\beta}^{q_{\beta}}$ with the power
\[ N= \sum_{p\,\mid\,m_{\beta}} q_{\beta} \;+ \sum_{p^2\,\mid\,m_{\beta}} \! q_{\beta} \;+\;\cdots \;+ \sum_{p^m\,\mid\,m_{\beta}} \! q_{\beta},\]
where $\beta$ varies over $\mathcal{H}_{\tilde{\alpha}}^{W}$. By assumption, for any $i \geq 0$,
\[ \sum_{\substack{\alpha \,:\, \alpha \in  \mathcal{H}_{\tilde{\alpha}}^V \\
\phantom{{}=\alpha \,:} \! p^i\,\mid\,m_{\alpha}}} \!\!r_{\alpha} \;\;\leq \sum_{\substack{\beta \,:\, \beta \in  \mathcal{H}_{\tilde{\alpha}}^{W} \\
\phantom{{}=\beta \,:} \! p^i\,\mid\,m_{\beta}}} \!\! q_{\beta}, \]
hence $M\leq N$. This shows that, for any prime $p$, the power of $p$ which appears in the decomposition of $e(V^{\mathbb{T}_{\tilde{\alpha}}})$ does not exceed the one which appears in the decomposition of $e(W^{\mathbb{T}_{\tilde{\alpha}}})$. Therefore $e(V^{\mathbb{T}_{\tilde{\alpha}}})$ divides $e(W^{\mathbb{T}_{\tilde{\alpha}}})$. Consequently, using the fact that $\tilde{\mathcal{A}} \subseteq \tilde{\mathcal{B}}$, $e(V) = \prod_{\alpha \in \tilde{\mathcal{A}}} e(V^{\mathbb{T}_{\tilde{\alpha}}})$ divides $e(W) = \prod_{\beta \in \tilde{\mathcal{B}}} e(W^{\mathbb{T}_{\tilde{\beta}}})$. A similar argument shows that the same thing happens for $e(V^H)$ and $e(W^H)$.
\end{remark}

The second-named author asked the following question in \cite[Problem 2.6]{Marzantowicz1}. Given orthogonal $\mathbb{S}^1$-representations $V$ and $W$ with $V^{\mathbb{S}^1}=W^{\mathbb{S}^1}=\{0\}$, is divisibility of $e(W)$ by $e(V)$ sufficient for the existence of an $\mathbb{S}^1$-equivariant map $S(V) \to S(W)$? The following example shows that the answer is negative in general.

\begin{example}
Let $V_1$ be the one-dimensional fixed-point free $\mathbb{S}^1$-representation. Define $V=2V_1^3 \oplus V_1^5$ and $W=V_1^{18} \oplus 2V_1^{5}$. Then $e(V)=45t^3$ divides $e(W)=450t^3$, but the existence of an $\mathbb{S}^1$-equivariant map $S(V) \to S(W)$ would violate Theorem \ref{sufficient for torus}, as $\dim V^{\mathbb{Z}_3}=2 > 1 = \dim W^{\mathbb{Z}_3}$.
\end{example}

\subsection{} 

It is known that if a group $G$ is not an extension of a finite $p$-group of exponent $p$ by a torus, then $G$ does not have the strong Borsuk--Ulam property, see \cite{MarzantowiczPacific}. It is an open problem whether every such extension enjoys the strong Borsuk--Ulam property; this is not even clear in the case $G=\mathbb{T}^k \times (\mathbb{Z}_p)^{\ell}$. (We note that the proof of \cite[Lemma~1.2]{MarzantowiczPacific} is incomplete and thus does not settle this last problem.)\medskip

\noindent\textbf{Conjecture.} \textit{A group $G$ has the strong Borsuk--Ulam property if and only if $G=\mathbb{T}^k\times (\mathbb{Z}_p)^{\ell}$.}\medskip

\noindent It remains to be verified that the following classes of groups do not have this property:
\begin{itemize}
\item non-abelian finite groups with exponent $p$, and
\item non-trivial extensions $0 \to \mathbb{T}^k \to G \to (\mathbb{Z}_p)^{\ell} \to 0$.
\end{itemize}


\noindent\textbf{Acknowledgments.}  
We wish to thank I. Nagasaki for pointing out a mistake in an earlier version of this paper.

The first and second authors have been supported by the National Science Centre under grants 2014/12/S/ST1/00368 and 2015/19/B/ST1/01458, respectively. The third author has been supported by DST INSPIRE Scheme IFA-11MA-01/2011 and DST-RSF project INT/RVS/RSF/P-2.


\end{document}